\documentclass{amsart}
\usepackage{amsmath,amssymb}
\usepackage[UKenglish] {babel}
\usepackage{tikz}
\usepackage{graphicx} 
\usepackage{anysize}
\usepackage{amssymb}
\usepackage{amsfonts}
\usepackage{mathrsfs}
\usepackage[all]{xy}
\usepackage{overpic}
\usepackage{fancyhdr}
\usepackage{hyperref}
\usepackage{ytableau}
\usepackage{multicol}
\usepackage{textcase}
\usepackage{setspace}
\usepackage{tcolorbox}
\usepackage{xy}
\xyoption{all}

\usepackage{amssymb,amsmath,amsthm,enumerate,graphicx}
\theoremstyle{definition}
\newtheorem{definition}{Definition}

\theoremstyle{plain}

\newtheorem{proposition}{Proposition}
\newtheorem{theorem}{Theorem}

\allowdisplaybreaks

\begin{document}

\title{Limiter Spaces: A Universal Extension for Limits of Real Sequences}
\author{Steven Lapp}
\address{S. Lapp, Department Of Mathematical and
Computational Sciences, University of Toronto Mississauga, 3359
Mississauga Road N., Mississauga, On L5L 1C6}
\email{steven.lapp@mail.utoronto.ca} 

\author{Marina  Tvalavadze}

\address{M. Tvalavadze, Department Of Mathematical and
Computational Sciences, University of Toronto Mississauga, 3359
Mississauga Road N., Mississauga, On L5L 1C6}
\email{marina.tvalavadze@utoronto.ca}

\begin{abstract}
We introduce the Limiter, a universal extension of the real numbers and of the limit functional that assigns a canonical limit in an enlarged space to every real sequence. Motivated by generalized summation methods such as Borel summation and Ramanujan’s assignments to divergent series, we require our extension to respect classical limits and assign limits in a way that depends only on the cluster points of a sequence and varies continuously when the cluster set is slightly modified.    
\end{abstract}
\maketitle
\textit{ Keywords:} divergent real sequences, limits, the extended real number line, a universal extension of real numbers. \\ \ \\
{\small \bf Mathematics Subject Classification 2010: 40A05, 54D35.}

 \section{Introduction}
Throughout the history of mathematics, many mathematicians have attempted to assign limits to real-valued divergent sequences and series using various techniques. Borel introduced Borel summation to assign values to divergent series \cite{Hardy}, and Ramanujan used the Euler–Maclaurin formula to obtain the famous assignment

\[
1 + \frac{1}{2} + \frac{1}{3} + \cdots = -\frac{1}{12}
\]
\cite{Cand}. These assignments, and Borel summation in particular, play a crucial role in quantum field theory \cite{Popov, Zinn}. However, this generalized notion of limit inherently disagrees with the usual topological definition of convergence. In this paper, we propose a universal extension of the real numbers and of the limit functional in which every sequence is assigned a canonical limit in an extended space. We refer to such an extension as the Limiter, and we define it in a manner analogous to the Stone–\v{C}ech\ compactification  and other categorical constructions \cite{Hindman, Ho, Hoehn}. We then formally define the Limiter and prove its existence.

\section{Cluster Points and Their Sets}

The real numbers $(\mathbb{R}, \leq)$ form a totally ordered set.We can extend this order by adjoining two new elements, $-\infty$ and $+\infty$,
such that $-\infty < x < +\infty$ for all $x \in \mathbb{R}$.
The resulting totally ordered set $\overline{\mathbb{R}} = \mathbb{R} \cup \{-\infty, +\infty\}$
is called the \emph{extended real number line} (or the \emph{extended real numbers}).\\ \ \\

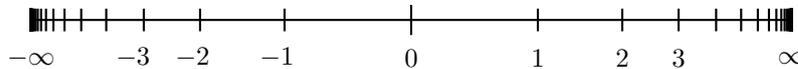
\begin{center}
\begin{tikzpicture}[thick]

  \draw (-5,0) -- (5,0);

  \draw (0,0.2) -- (0,-0.2);
  \node at (0,-0.5) {$0$};

  \def\leftEnd{-5}
  \def\rightEnd{5}
  \pgfmathsetmacro{\leftClosest}{0}
  \pgfmathsetmacro{\rightClosest}{0}

  \def\iterations{100}

  \foreach \i in {1,...,\iterations} {
    
    \pgfmathsetmacro{\newLeft}{\leftEnd + 2*(\leftClosest - \leftEnd) / 3}
    \draw (\newLeft,0.15) -- (\newLeft,-0.15);

    \ifnum\i<4
      \pgfmathtruncatemacro{\labelLeft}{-\i}
      \ifnum\i=3
        \node[xshift=-0.14cm] at (\newLeft,-0.5) {$\labelLeft$}; 
      \else
        \node[xshift=-0.1cm] at (\newLeft,-0.5) {$\labelLeft$}; 
      \fi
    \fi
    \xdef\leftClosest{\newLeft}

    \pgfmathsetmacro{\newRight}{\rightEnd + 2*(\rightClosest - \rightEnd) / 3}
    \draw (\newRight,0.15) -- (\newRight,-0.15);

    \ifnum\i<4
      \pgfmathtruncatemacro{\labelRight}{\i}
      \node at (\newRight,-0.5) {$\labelRight$};
    \fi
    \xdef\rightClosest{\newRight}
  }

  \node at (-5,-0.5) {$-\infty$};
  \node at (5,-0.5) {$\infty$};

\end{tikzpicture}
\end{center}

It is standard to endow the generalized extended real numbers with the order topology induced by their natural order.
The resulting space has three notable properties:

\hspace{5mm} (1) The real numbers $\mathbb{R}$ with their standard topology form a subspace of $\overline{\mathbb R}.$

\hspace{5mm} (2) The order topology on any linearly ordered set is Hausdorff, so  $\overline{\mathbb{R}}$ is Hausdorff.

\hspace{5mm} (3) The extended real line  $\overline{\mathbb{R}}$  with the order topology is compact. 
\par\medskip
Since $\mathbb{R}$ with its usual topology is a subspace of
$\overline{\mathbb{R}} = [-\infty,\infty]$, every real cluster point of a
sequence $(x_n)$ in $\mathbb{R}$ is also a cluster point of $(x_n)$ viewed
as a sequence in $\overline{\mathbb{R}}$. We will denote by
$\Lambda(x_n),$ the set of all cluster points of $(x_n)$ in
$[-\infty,\infty].$


\begin{definition} Let $(x_n)$ be a sequence of real numbers. Then $L \in [-\infty,\infty]$ is called a
\emph{cluster point} of $(x_n)$ if there exists a subsequence $(x_{n_k})$ 
that converges to $L$. The set of all such cluster points of $(x_n)$ is denoted by $\Lambda (x_n):$
 $$\Lambda(x_n) =\{\hspace{1mm} L \in \overline{\mathbb R} \hspace{1mm} | \hspace{1mm} \text{ there exists a subsequence} \hspace{1mm} (x_{n_k}) \hspace{1mm} \text{converging to} \hspace{1mm} L \hspace{1mm} \}.$$ 
\end{definition}
\par\medskip
\noindent The compactness of $\overline{\mathbb{R}}$ asserts that every sequence of real numbers has a cluster point. That is, 
$\Lambda(x_n) \neq \emptyset$ for any sequence of real numbers $(x_n)$.

  

\begin{definition}
Let $(x_n)$ and $(y_n)$ be sequences of real numbers. 
We say that $(x_n)$ and $(y_n)$ are \emph{similar} if they have the same
set of cluster points, that is, if $\Lambda(x_n) = \Lambda(y_n)$.
In this case we write $(x_n) \sim (y_n)$.
\end{definition}

Since similarity is defined in terms of set equality, we can conclude that $\sim$ is an 
equivalence relation on the set of real sequences. The equivalence class of a sequence $(x_n)$ will be denoted by $[\hspace{0.5mm} (x_n) \hspace{0.5mm} ]$.

In our first theorem, we establish a connection between  similarity and the assignment of limits. 

\begin{theorem} Let $(x_n), (y_n)$ be two sequences in $\mathbb{R}$ and suppose their limits in $\overline{\mathbb{R}}$ exist.
Then, $\displaystyle \lim_{n \to \infty}x_n = \lim_{n \to \infty}y_n$ if and only if $(x_n) \sim (y_n)$. 

\end{theorem}
\begin{proof}
  The first direction follows from $[-\infty,\infty]$ being a Hausdorff space. Recall that a
convergent sequence in Hausdorff spaces has a unique cluster point. As such, 
the following sets are equal whenever  $\displaystyle \lim_{n \to \infty}x_n = \lim_{n \to \infty}y_n$
$$\Lambda(x_n) = \{\displaystyle \lim_{n \to \infty}x_n\} = \{\displaystyle \lim_{n \to \infty}y_n\} = \Lambda(y_n)$$
By the same reasoning, the converse holds whenever $(x_n) \sim (y_n)$. This proves the result.

\end{proof}

Theorem 1 tells us that, among sequences that converge in $\overline{\mathbb{R}}$, sequences with the same cluster points are assigned the same limit. Therefore, when assigning limits to divergent sequences in our extension, we should do so in a way that is consistent with this theorem. Beyond this property, there are additional conditions that we would like our extension to satisfy.
Recall that every real number is the limit of some sequence of rational numbers. When we introduce a new point 
$L$ into our extension of $\mathbb R$, we require that 
$L$ be the limit of some sequence of real numbers. In this way, the reals form a dense subset of the extension and the extended limits remain consistent with the original notion of limit. This is condition (2) in the definition below.

For condition 3, we want to ensure that the limit operation is stable in our extended 
topology. By stability of the limit, we mean two sequences with ``close" cluster point sets also have ``close" limits in the extension. Consequently, if the limit of a real number sequence $(x_n)$ belongs to
an open set $O$ of our extension, then we can find a smaller neighborhood $V\subseteq O$ that  
consists of only limits of sequences with cluster points sufficiently close to the cluster set of $(x_n)$. This will be 
formalized later. 

Finally, there may be multiple extensions satisfying the first three conditions. We therefore require our model of the axioms to be universal in the sense that, for any other representation of these axioms, there exists a unique mapping from our model into it. With this goal in mind, we now introduce some preliminary definitions.

\begin{definition} Let $X$, $Y$ be  topological spaces and suppose $\varphi: X \to Y$ is a continuous function. The image of $\varphi$, denoted by 
$Im(\varphi),$ is the set of all points $y \in Y$ achieved by applying $\varphi$ to all points in $X$. It is 
endowed with the subspace topology induced by $Y.$
   
\end{definition}

\begin{definition} Let $X$ and $Y$ be topological spaces. 
A map $\sigma: X\to Y$ is called an embedding if $\sigma$ is continuous and injective, and the induced map  $\sigma: X\to \sigma(X)$, where $\sigma(X)$ is equipped with the subspace topology from $Y$, is a homeomorphism.
\end{definition}
If an embedding of $X$ into $Y$ exists, then there is a subspace of 
$Y$ that is isomorphic to  $X$. For this reason, we call $Y$ an \emph{extension} of $X$.

Next, we introduce  notation for the set of 
all sequences of real numbers and their equivalence classes.

 The set of all real-number sequences $(x_n)$ is denoted by $Seq(\mathbb{R})$. 
The set of their equivalence classes under $\sim$ is denoted by $Seq(\mathbb{R})/\hspace{-0.5mm}\sim$.

With these preliminary definitions outlined, we now formalize our extension.

\begin{definition}

The \emph{Limiter} consists of a topological space $\mathcal X$, an embedding $i: \overline{\mathbb{R}} \to \mathcal X$
and a function $\displaystyle \underset{n\to \infty}{Lim}: Seq(\mathbb{R}) \to \mathcal X$ satisfying the following axioms for all pairs of real  
sequences $(x_n),(y_n) \in Seq(\mathbb{R})$, a point $\alpha \in \mathcal X$ and an open set $O$ in $\mathcal X$.
\begin{itemize}
    \item[(1)] The sequence $(i(x_n))$ converges to $\underset{n\to \infty}{Lim}(x_n)$ in $\mathcal X;$
    \item[(2)] If $(x_n)$ is convergent in $\overline{\mathbb{R}},$ then $ \displaystyle i( \hspace{1mm} \lim_{n\to\infty} x_n \hspace{1mm}) = \underset{n\to \infty}{Lim}(x_n);$
    \item[(3)] $\underset{n\to \infty}{Lim}(x_n) = \underset{n\to \infty}{Lim}(y_n)$ if and only if $(x_n) \sim (y_n);$
    \item[(4)] There exists $(z_n) \in Seq(\mathbb{R})$ such that $\underset{n\to \infty}{Lim}(z_n)= \alpha;$
    \item[(5)] If $\underset{n\to \infty}{Lim}(x_n) \in O,$ then there exists an open set $V \subseteq O$ containing $\underset{n\to \infty}{Lim}(x_n)$ 
     such that $\underset{n\to \infty}{Lim}(y_n) \in V$ if and only if $i(\Lambda (y_n)) \subseteq V .$ 
\end{itemize}

Finally, if $\mathcal Y$ is a topological space, $j: \overline{\mathbb{R}} \to \mathcal Y$ is an embedding and $\mathcal{L}: Seq(\mathbb{R}) \to \mathcal Y $ is a
function satisfying axioms (1)-(5), then there exists a unique continuous function 
$T: \mathcal X \to \mathcal Y$ such that $T \circ \hspace{0.5mm} i = j$ and $T \circ \hspace{0.5mm} \underset{n\to \infty}{Lim} = \mathcal{L}$. The
universality of the Limiter is represented by the following commutative diagram:

\begin{center}
\begin{tikzpicture}[thick, ->, scale=1.15, every node/.style={transform shape}]

  \node (X) at (0,2) {$X$};
  \node (R) at (-3,0) {$\overline{\mathbb{R}}$};
  \node (Y) at (0,0) {$Y$};
  \node (SeqR) at (3,0) {$Seq(\mathbb{R})$};

  \draw[dashed] (X) -- node[pos=0.6, right, xshift=2pt] {$T$} (Y);
  \draw (R) -- node[pos=0.6, left, xshift=-8pt] {$i$} (X);
  \draw (R) -- node[pos=0.6, below, yshift=-2pt] {$j$} (Y);
  \draw (SeqR) -- node[pos=0.6, right, xshift=12pt] {$\underset{n\to \infty}{Lim}$} (X);
  \draw (SeqR) -- node[pos=0.6, below, yshift=-2pt] {$\mathcal{L}$} (Y);

\end{tikzpicture}
\end{center}

\end{definition}

Although we did not explicitly discuss axioms 1 and 2 in our Limiter motivation, we must ensure that our assignment of limits in the extension agrees with the limits already defined in the real numbers; this is the content of axiom 2. Moreover, we must ensure that the points selected as limits are actually accumulation points of the sequence in $\mathcal X,$
 as required by axiom 1. In category theory, one can define mathematical objects in terms of the Limiter, and such objects do indeed exist; this will be the subject of the next section.

\section{Construction of the Limiter}

To construct the Limiter of the real numbers, we introduce an operation on open subsets 
of the generalized real numbers: $[-\infty,\infty].$

\begin{definition} 
Given an open subset $O$ of  $[-\infty,\infty]$. Then
$$O^* = \{\hspace{1mm} [\hspace{0.5mm} (x_n) \hspace{0.5mm}] \in Seq(\mathbb{R})/ \hspace{-1mm} \sim \hspace{1mm} \hspace{1mm} | \hspace{1mm} \Lambda(x_n) \subseteq O\}.$$ 

\end{definition}
   
\noindent With this concept established,  we introduce some properties of the map $O \mapsto O^*$. \\

\noindent First, we note that every sequence of real numbers $(x_n) \in Seq(\mathbb{R})$ has at  least one cluster  point in $\overline{\mathbb{R}}$, that is, $\Lambda(x_n)\neq \emptyset. $ This condition implies that no
equivalence class $[\hspace{0.5mm}(x_n)\hspace{0.5mm}]$ is contained in $\emptyset^*$; otherwise,  $ \Lambda(x_n) \subseteq \emptyset$, a contradiction. Furthermore, the sets of cluster points by definition are subsets of $\overline{\mathbb{R}}$. As such, every equivalence class must be a member of $(\overline{\mathbb{R}})^*$. These observations yield the following.\\

$\hspace{-4mm} \mathbf{Property \hspace{1mm} 1}$: $ (\overline{\mathbb{R}})^*= [-\infty,\infty]^* = Seq(\mathbb{R})/ \hspace{-1mm} \sim \hspace{1mm}$ and $\emptyset^* = \emptyset.$\\


The transitive property of the inclusion relation asserts that $\Lambda(x_n) \subseteq V$ whenever $\Lambda(x_n) \subseteq U$ and $U \subseteq V$. In this expression, $(x_n)$ is a sequence of real numbers
while $U,V$ are open subsets of  $[-\infty,\infty]$. Consequently, $U^* \subseteq V^*.$\\

$\hspace{-4mm} \mathbf{Property \hspace{1mm} 2}$: For any pair of open sets $U,V$ in $[-\infty,\infty]$ with $U \subseteq V$, we have that  $U^* \subseteq V^*.$\\

Likewise, with property 2, we can use set theoretical concepts to gain further insight into the operation $*$. For this observation, we can consider a finite collection of 
open subsets of $\overline{\mathbb{R}}$ denoted by $\mathcal{O} = \{O_i\}_{i=1}^{n}$.
Let $I$ represent the intersection of all open sets in the collection $\mathcal{O}$.

Given a sequence $(x_n) \in Seq(\mathbb{R})$, notice $\Lambda(x_n) \subseteq I$ if and only if $\Lambda(x_n) \subseteq O_i$ for all 
$1 \leq i \leq n$. Once again, our definition of the operation $*$ tells us that $\displaystyle [\hspace{0.5mm} (x_n) \hspace{0.5mm} ] \in I^*$ if 
and only if $[\hspace{0.5mm} (x_n) \hspace{0.5mm} ]  \in O_i^*$ for all $1 \leq i \leq n$. Therefore, the map $O \mapsto O^*$ preserves
finite intersections. We refer to this as property 3. \\

$\hspace{-4mm} \mathbf{Property \hspace{1mm} 3}$: Let $\displaystyle \{O_i\}_{i=1}^k$ be a collection of open sets in $\overline{\mathbb{R}}$, then $\displaystyle \left(\hspace{0.5mm} \bigcap_{i=1}^{k}O_i \hspace{0.5mm}\right)^* = \bigcap_{i=1}^{k} O_i^*.$ \\

Properties 1 and 3 implies that the collection $\tau^* = \{O^* \hspace{1mm} | \hspace{1mm} O \in \tau \hspace{1mm}\}$ is a \emph{basis} for a
topology in $Seq(\mathbb{R})/\hspace{-1mm}\sim$. We assume going forward that $Seq(\mathbb{R})/\hspace{-1mm}\sim$ is equipped with  this topology. The preceding observations can be summarized in the following theorem.

\begin{theorem}

The collection $\tau^* = \{ \hspace{1mm} O^* \hspace{1mm} | \hspace{1mm} O \hspace{1mm} \text{open in} \hspace{1mm} [\infty,\infty] \hspace{1mm} \}$ is the basis for a topology on $Seq(\mathbb{R})/\hspace{-1mm}\sim$.

\end{theorem}

Now that we have constructed a topological space, we now define an embedding of the extended reals into our space. We do so by defining the inclusion map.

\begin{definition} The inclusion map $i:\overline{\mathbb{R}} \to Seq(\mathbb{R})/\hspace{-1mm}\sim$ is the function sending each real 
number $x \in \mathbb{R}$ to the equivalence class of its constant sequence representation $x_n = x$ for any $n\in \mathbb N.$
Additionally, we map $\infty$ to the equivalence class of the sequence $(n)$ and $-\infty$ to the 
class of $(-n)$. In other words, it is the function $i: \overline{\mathbb{R}} \to Seq(\mathbb{R})/ \hspace{-1mm} \sim \hspace{1mm} $ given by
   
\end{definition}

$$i(x)=
\begin{cases}

\vspace{5mm}

\hspace{3mm} [\hspace{1mm}(x)\hspace{1mm}] \hspace{6mm} \text{if} \hspace{2mm} x \in \mathbb{R} \\

\vspace{5mm}

\hspace{3mm} [\hspace{1mm}(n) \hspace{1mm}] \hspace{8mm} \text{if} \hspace{2mm} x = \infty \\

\vspace{3mm}

\hspace{3mm} [\hspace{1mm}(-n) \hspace{1mm}] \hspace{5mm} \text{if} \hspace{2mm} x = -\infty \\
\end{cases}$$

Notice that $\Lambda(x) = \{x\}$ when $x \in \mathbb{R}$ while $\Lambda(n) = \{\infty\}, \Lambda(-n) = \{-\infty\}$. Therefore, by definition of $i,$ we map each $x \in \overline{\mathbb{R}}$ to the equivalence class of a sequence with $x$ as its unique cluster point.
Therefore, for any open set $O$ of $\overline{\mathbb{R}}$, we have that $$i(x) \in O^* \iff \{x\} \subseteq O \iff x \in O.$$ This is property 4 of our 
operation $*$. 
\par\medskip

$\hspace{-4mm} \mathbf{Property \hspace{1mm} 4}$: For any $x \in \overline{\mathbb{R}}$ and an open subset $O \subseteq \overline{\mathbb{R}}$, $i(x) \in O^*$ 
if and only if $x \in O.$ \\

Using property 4, we will prove that the inclusion map is an embedding. This leads to the next two propositions.

\begin{proposition} The inclusion map $i: \overline{\mathbb{R}} \to Seq(\mathbb{R})/ \hspace{-1mm} \sim \hspace{1mm},$ given by $i(x) = [\hspace{0.5mm} (x) \hspace{0.5mm}] $ is an injection.  
\end{proposition}
\begin{proof} 
 By construction, $i(x)$ is the set of all real sequences with $x$ in $[-\infty ,\infty]$ as its only cluster point. Assuming that $i(x) = i(y),$ we conclude that $x$ and $y$ are both the unique cluster points of any sequence $(a_n) \in i(x) = i(y)$. It follows that $x = y$.


\end{proof}

\begin{proposition} The inclusion map $i: \overline{\mathbb{R}} \to Seq(\mathbb{R})/ \hspace{-1mm} \sim \hspace{1mm},$ given by $i(x) = [\hspace{0.5mm} (x) \hspace{0.5mm}] $ is an embedding. 

\end{proposition}

\begin{proof}
  
Consider a basic set $O^*$ where $O$ is open in $\overline{\mathbb{R}}$. Then, the following sets are equal by Property 4:
$$i^{-1}(O^*) = \{ x \in \overline{\mathbb{R}} \hspace{1mm} | \hspace{1mm} i(x) \in O^*\} = \{ x\in \overline{\mathbb{R}} \hspace{1mm} | \hspace{1mm} \{x\} \subseteq O\} = \{x\in \overline{\mathbb{R}} \hspace{1mm} | \hspace{1mm}  {x} \in O\} = O.$$

\noindent This implies  that $i$ is continuous. 

Since $i(x) \in O^*$ if and only if  $x \in O$, we have the following:

\begin{centering}
  
$i(O) = \{i(x) \hspace{1mm} | \hspace{1mm} x \in O \} = \{i(x) \hspace{1mm} | \hspace{1mm} x \in \overline{\mathbb{R}} \hspace{1mm} \text{and} \hspace{1mm} x \in O \}$
\vspace{3mm}

\hspace{40mm} $= \{i(x) \hspace{1mm} | \hspace{1mm} x \in \overline{\mathbb{R}} \hspace{1mm} \text{and} \hspace{1mm} i(x) \in O^* \}$

\vspace{3mm}

\hspace{5mm} $= O^* \cap i(\overline{\mathbb{R}}).$ 

\end{centering}

Therefore, the image of every open set is also open in the subspace $i(\overline{\mathbb{R}})$. Thus,
restricting the codomain of $i$ to its image results in a homemorphism. This proves the result.  
\end{proof}

In our next proposition, we show that every sequence of real numbers $(x_n) \in Seq(\mathbb{R})$ is convergent 
in $Seq(\mathbb{R})/\hspace{-1mm}\sim$.

Let  $\varphi_{\sim}$ be the quotient map of the equivalence relation $\sim,$ that is, 
$$\varphi_{\sim}: Seq(\mathbb{R}) \to Seq(\mathbb{R})/\hspace{-1mm}\sim,\quad \varphi(x_n) = [\hspace{0.5mm}(x_n)\hspace{0.5mm}].$$

Furthermore, we demonstrate that the triple $(Seq(\mathbb{R})/\hspace{-1mm}\sim,\hspace{0.5mm} i,\hspace{0.5mm} \varphi_{\sim})$ is the \emph{Limiter}.
\par\medskip

\begin{theorem} For any sequence $(x_n) \in Seq(\mathbb{R})$, the sequence $(i(x_n))$  converges to $[\hspace{0.5mm}(x_n)\hspace{0.5mm}]$ in $Seq(\mathbb{R})/\hspace{-1mm}\sim.$
Furthermore, the triple $(Seq(\mathbb{R})/\hspace{-1mm}\sim,\,i,\,\varphi_{\sim})$ is the 
Limiter.  
  
\end{theorem}
\begin{proof}
 First, we show that the sequence $(i(x_n))$ converges to 
 $[\hspace{0.5mm}(x_n)\hspace{0.5mm}]$ for any sequence $(x_n) \in Seq(\mathbb{R}).$
For the sake of contradiction, let us assume that we can find a sequence $(x_n) \in Seq(\mathbb{R})$ so that $(i(x_n))$ does not converge to $[\hspace{0.5mm}(x_n)\hspace{0.5mm}]$. Then there exists a basic set $O^*$ such that $[\hspace{0.5mm}(x_n)\hspace{0.5mm}] \in O^*$
and a subsequence $(x_{n_k})$ with the property that $i(x_{n_k}) \not \in O^*$ for all $n_k \in \mathbb{N}$. 
Since a cluster set of any sequence in $\overline{\mathbb{R}}$ is non-empty, we can find a subsequence of $(x_{n_k})$ which is convergent in $\overline{\mathbb{R}}.$
Denote this subsequence by $(x_{n_{k_w}})$ and its limit by $L \in \overline{\mathbb{R}}$.
Due to being a subsequence of $(x_n)$, $L \in \Lambda(x_n) \subseteq O$ by definition of $O^*$.

Since the inclusion map $i: \overline{\mathbb{R}} \to Seq(\mathbb{R})/ \hspace{-1mm} \sim \hspace{1mm}$ is continuous and, in particular, sequentially continuous,  we have that $i(x_{n_{k_w}})$ converges to 
$i(L)$. Consequently, $i(L)$ being a member of an open set $O^*$ allow us to find a natural number $N \in \mathbb{N}$ with the property that $i(x_{n_{k_w}}) \in O^*$ whenever $n_{k_w} > N$. 
This is a contradiction. Therefore, $i(x_n)$ converges to $[ \hspace{0.5mm}(x_n) \hspace{0.5mm} ] = \varphi_{\sim}(x_n)$. This establishes the
first axiom of the Limiter. \\ 

We next verify that $(Seq(\mathbb{R})/\hspace{-1mm}\sim,\,i, \varphi_{\sim})$ satisfies the remaining axioms, that is,
\par\medskip
\hspace{5mm} (2) If $(x_n)$ is convergent in $\overline{\mathbb{R}}$, then $ \displaystyle i( \hspace{1mm} \lim_{n\to\infty} x_n \hspace{1mm}) = \varphi_{\sim}(\displaystyle \lim_{n \to \infty}(x_n)) = [\hspace{0.5mm}(\displaystyle \lim_{n \to \infty}(x_n))\hspace{0.5mm}];$

\vspace{5mm}

\hspace{5mm} (3)  $[\hspace{0.5mm}(x_n)\hspace{0.5mm}] = \varphi_{\sim}(x_n) = \varphi_{\sim}(y_n) = [\hspace{0.5mm}(y_n)\hspace{0.5mm}]$ if and only if $(x_n) \sim (y_n);$

\vspace{5mm}

\hspace{5mm} (4) For all $\alpha \in Seq(\mathbb{R})/\hspace{-1mm}\sim$, there exists $(z_n) \in Seq(\mathbb{R})$ such that $[\hspace{0.5mm} (z_n) \hspace{0.5mm}] = \varphi_{\sim}(z_n) = \alpha$;

\vspace{5mm}

\hspace{5mm} (5) For any open set $\widetilde{O}$ with $\phi_{\sim}(x_n) \in \widetilde{O}$ there exists a sub-neighborhood $W \subseteq \widetilde{O}$ of 

\vspace{3mm}

\hspace{11mm} $\varphi_{\sim}(x_n)$ such that $[(y_n)] \in W$ if and only if $\Lambda(x_n) \subseteq i^{-1}(W).$

\vspace{7mm}

In the above axioms, we denote $\underset{n\to \infty}{Lim}$ by $\varphi_{\sim}$ for convenience. We note that conditions (2)-(4) follow from the construction of $Seq(\mathbb{R})/\sim$. It remains to prove the fifth. 

Let $\widetilde{O}$ be an open set containing 
$ \varphi_{\sim}(x_n) = [\hspace{0.5mm}(x_n)\hspace{0.5mm}] \in \widetilde{O}$. Then $\widetilde{O}$ is the union of basic sets by definition of our topology. As such, one of the basic sets $W$ in the union must contain our specified equivalence class $[\hspace{0.5mm}(x_n)\hspace{0.5mm}]$. For $W$ there exists $V \subseteq [-\infty,\infty]$ such that  $W=V^{*}$. We 
have therefore found an open set $V \subseteq [-\infty,\infty]$ such that $[(x_n)] \in V^* \subseteq \widetilde{O} $. 
We recall $V^* = \{[(x_n)] \hspace{1mm} | \hspace{1mm} \Lambda(x_n) \subseteq V \}$. By our proof of $i$ being 
an embedding, $i^{-1}(V^*) = V$. Consequently, $\varphi_{\sim}(y_n) = [(y_n)] \in V^* $ if and only if
$\Lambda(y_n) \subseteq V = i^{-1}(V^*)$. This proves the final axiom. Therefore, $(Seq(\mathbb{R})/\hspace{-1mm}\sim,\,i, \varphi_{\sim})$ satisfies the axioms of being a Limiter. \\

It remains to show that it is universal. 
To prove universality, we consider an arbitrary triple $(\mathcal {Y},\,j,\, {L})$ where $j:\overline{\mathbb{R}} \to \mathcal Y$ and ${L}: Seq(\mathbb{R}) \to \mathcal Y$ satisfying axioms of the Limiter. Let $\mathcal X$ denote $Seq(\mathbb{R})/\hspace{-1mm}\sim$. Then we define $T: \mathcal X \to \mathcal Y$ by
$$T([\hspace{0.5mm}(x_n)\hspace{0.5mm}]) = {L}(x_n).$$
By axiom 3,  $T$ is not only a well defined function but also injective. Furthermore, axiom 4 implies that $T$ is surjective. It is therefore a bijection.

We will show that $T$ is the unique continuous function satisfying the following:

$$T \circ i = j \hspace{2mm} \text{and} \hspace{2mm} T \circ \varphi_{\sim} = {L}$$ 

We begin by proving continuity of $T$.
Let $O$ be an open set in $\mathcal Y$. Then, for every point $\alpha \in O$, there exists an open subset $V_{\alpha} \subseteq O$ such that $\alpha \in V_{\alpha}$ and  for all sequences 
$(y_n) \in Seq(\mathbb{R}):$  ${L}(y_n) \in V$ if and only if  $j(\Lambda(y_n)) \subseteq V_{\alpha}$. This is a consequence of axioms 4 and 5. 
Under the same constraints, continuity of the embedding $j: \overline{\mathbb{R}} \to \mathcal{Y}$, asserts that $j^{-1}(V_{\alpha})$ is an open subset of $\overline{\mathbb{R}}$. Due to
this property, we have that $(j^{-1}(V_{\alpha}))^*$ is basic set in $\mathcal X= Seq(\mathbb{R})/\hspace{-1mm}\sim$ for any $\alpha \in O$. Taking the image of these basic sets under $T$ gives us the following: 
\[
(j^{-1}(V))^* = \{ \hspace{1mm} [(x_n)]\in \mathcal{X} \hspace{1mm} | \hspace{1mm} [(x_n)] \in  (j^{-1}(V)))^* \hspace{1mm} \} = 
\{[(x_n)]\in \mathcal{X} \hspace{1mm} | \hspace{1mm} \Lambda(x_n) \subseteq j^{-1}(V)\},
\]
\[
\begin{aligned}
T\bigl((j^{-1}(V_{\alpha}))^*\bigr) 
  &= \{\, T([x_n]) \;\big|\; \Lambda(x_n) \subseteq j^{-1}(V_{\alpha}) \,\} \\[4pt]
  &= \{\, {L}(x_n) \;\big|\; j(\Lambda(x_n)) \subseteq V_{\alpha} \,\} \\[4pt]
  &= \{\, {L}(x_n) \;\big|\; {L}(x_n) \in V_{\alpha }\,\}= V_{\alpha}
\end{aligned}
\]

This shows  that $T\bigl((j^{-1}(V_{\alpha}))^*\bigr) = V_{\alpha}$ for any $\alpha \in O$. Due to $T$ being a bijection, the above is equivalent to $T^{-1}(V_{\alpha}) = (j^{-1}(V_{\alpha}))^{*}$ for any $\alpha \in O$. It follows that

$$\displaystyle T^{-1}(O) = T^{-1} (\bigcup_{\alpha \in O} V_{\alpha}) = \bigcup_{\alpha \in O} T^{-1}(V_{\alpha}) = \bigcup_{\alpha \in O} (j^{-1}(V_{\alpha}))^*$$

This proves that $T^{-1}(O)$ can be represented as the union of basic sets, proving continuity of $T$.

We now prove that $T \circ i = j$ and $T \circ \phi_{\sim} = {L}.$
Consider $x \in \overline{\mathbb{R}}$. Then, by construction of $i$ and 
Theorem 1, we have that $\underset{n\to \infty}{Lim}(x_n) = x$ for all sequence in the equivalence class $i(x)$.
Consequently, the following are equal by axiom 2:

$$T(i(x))= T([(x)])={L}((x))=j(\hspace{0.5mm} \underset{n\to \infty}{lim}(x) \hspace{0.5mm}) = j(x)$$

Due to the above argument holding for all $x \in \overline{\mathbb{R}}$, we conclude that $T \circ i  = j$. Similarly, 
we can conclude that $T \circ \varphi_{\sim} = {L}$

$$T( \hspace{0.5mm}\phi_{\sim}(x_n)\hspace{0.5mm}) = T(\hspace{0.5mm}[\hspace{0.5mm}(x_n)\hspace{0.5mm}]\hspace{0.5mm}) = {L}(x_n)$$

\noindent Therefore,  $T \circ i = j$ and $T \circ \varphi_{\sim} = {L}$. 

To conclude the proof of universality, we show 
that $T$ is unique.

Let $T': \mathcal{X}  \to \mathcal{Y}$ be a continuous function satisfying our conditions. Then, for
any equivalence class $[(x_n)] \in \mathcal X$, the following holds
$$T'(\hspace{0.5mm}[\hspace{0.5mm}(x_n)\hspace{0.5mm}]\hspace{0.5mm}) = T'(\hspace{0.5mm} \varphi_{\sim}(x_n) \hspace{0.5mm} ) = (T'\circ \varphi_{\sim})(x_n) =  {L}(x_n) = T(\hspace{0.5mm}[\hspace{0.5mm}(x_n)\hspace{0.5mm}]\hspace{0.5mm}).$$
This proves 
uniqueness of the function $T$ and in turn universality. 

We have successfully constructed the Limiter of the real numbers. 
\end{proof}

Building on this foundation, our next objective is to study the intrinsic topological properties of the Limiter—examining continuity, determining which countability axioms are satisfied, and identifying which separation axioms hold. Although the Limiter ensures that every sequence has a canonical limit, we have not yet defined canonical limits for general nets in this space. This warrants further investigation. Since we have defined an extension of the real numbers, we should also define continuous extensions of the addition and multiplication operations and examine their algebraic properties. Once a theory of addition and multiplication has been established, it may be worthwhile to investigate twin primes by examining the limit of the associated oscillating sequence: $p, p+2, p, p+2, \ldots$ where $p$ is a prime number.

\end{document}